\documentclass[reqno]{amsart}
\usepackage{hyperref}

\begin{document}
\title[\hfilneg \hfil Existence of solutions]
{Existence and Multiplicity of solutions for Kirchhoff type
equations in Physical Education}

\author[Amirreza Kiaroosta, Seyyed Sadegh Kazemipoor \hfil \hfilneg]
{Amirreza Kiaroosta, Seyyed Sadegh Kazemipoor}

\address{Amirreza Kiaroosta \newline
} \email{amirkiaroosta306@gmail.com}

\address{Seyyed Sadegh Kazemipoor (corresponding author) \newline
Mathematics Department, Sharif University of Technology \\
Public University in Tehran, Iran} \email{s.kazemipoor@umz.ac.ir}

\thanks{}
 \keywords{ Kirchhoff type
equation; Dirichlet problem; Physical Education; critical point;
\hfill\break\indent Mountain Pass Theorem}

\begin{abstract}
The fact that potentially skilled, but biologically later-maturing
athletes are less likely to be selected into talent development
programmes can represent a failure of Talent Identification in
sports. In this article, we prove the existence of solutions for
Kirchhoff type equations with Dirichlet boundary-value condition. We
use the Mountain Pass Theorem in critical point theory, without the
(PS) condition.
\end{abstract}

\maketitle \numberwithin{equation}{section}
\newtheorem{theorem}{Theorem}[section]
\newtheorem{lemma}[theorem]{Lemma}
\newtheorem{remark}[theorem]{Remark}
\newtheorem{definition}[theorem]{Definition}
\allowdisplaybreaks

\section{Introduction and statement of main results}

Consider the Kirchhoff type problem
\begin{equation}\label{Eq.(5.1)} 
\begin{gathered}
-(a+b\int_\Omega |\nabla u|^2dx)\Delta u=f(x,u), \quad\text{in } \Omega, \\
u=0,  \quad\text{on } \partial \Omega,
\end{gathered}
\end{equation}
where $\Omega $ is a smooth bounded domain in $ {\mathbb {R}}^N$,
$a,b>0$, and $f(x,t): \overline{\Omega}\times {\mathbb R}$ is a
continuous real function and satisfies the subcritical condition
\begin{equation}\label{8.1}
 |f(x,t)|\leq C(|t|^{p-1}+1)\quad \text{for some }
 2<p<2^*=\begin{cases}\frac{2N}{N-2},& N\geq 3,\\
\infty, & N=1,2,
\end{cases}
 \end{equation}
where $C$ denotes some positive constant.

It is pointed out in \cite{MB} that the similar nonlocal problems
model several physical and biological systems where $u$ describes a
process which depends on the average of itself, for example that of
the population density.

Problem \eqref{Eq.(5.1)} is related to the stationary analogue of
the
 Kirchhoff equation
$$
u_{tt}-\Big(a+b\int_\Omega|\nabla u|^2dx\Big)\Delta u=g(x,u),
$$
proposed by Kirchhoff in \cite{GK} as an extension of the classical
d'Alembert wave equation for free vibrations of elastic strings.
Kirchhoff's model takes into account the changes in length of the
string produced by transverse vibrations. It received great
attention only after Lions \cite{JL} proposed an abstract framework
for the problem.

Positive solutions are considered by authors, such as Alves et al
\cite{CO}, Ma and Rivera \cite{TF}, Cheng and Wu \cite{CX}, Yang and
Zhang \cite{YJ}. Problems on the unbounded domain $ {\mathbb {R}}^3$
have also been considered by authors, such as Jin and Wu \cite{JXW},
Nie and Wu \cite{NX}, Wu \cite{XW}, Liu and He \cite{LH}, Li et al
\cite{LL}, Li and Ye \cite{GL}. And more recently the concentration
behavior of positive solutions has
 been studied by He and Zou \cite{HZ}, Wang et al \cite{WT}, He et al \cite{hl}.
A result with Hartree-type nonlinearities can be found in L\"{u}
\cite{df}. Ground state nonlinear with critical growth is considered
in He and Zou \cite{HZ2}. The readers may consult Bernstein
\cite{SB} and Poho\v{z}aev \cite{SI}, Sun and Tang \cite{JJS}, Chen
et al \cite{CY}, Cheng \cite{BC}, Perera and Zhang \cite{KP, ZT},
Mao and Zhang \cite{AM}, Sun and Liu \cite{SJ} and the references
therein, for more information on this problem.

There are many solvability conditions for problem \eqref{Eq.(5.1)}
with $f$, like the asymptotical linear case (at infinity)  in
\cite{YJ} and people are more interested in the superlinear case (at
infinity):
\begin{itemize}
\item[(S1)] there exists $\theta \geq 1$ such that $\theta G(t)\geq G(st)$
for all $t\in \mathbb {R}$ and $s \in [0, 1]$, where
$G(t)=f(t)t-4F(t)$ (see \cite{JJS});

\item[(S2)] $\lim_{|t|\to\infty} G(t)=\infty$ and there exist
$ \sigma> \max \{1, N/2\}$ and $C>0$ such that $|f (t)|^\sigma\leq
CG(t)|t|^\sigma$ for $|t|$ large (see \cite{AM}); or some limitation
forms,

\item[(S3)] $\lim_{|t|\to \infty} [f(t)t-4F(t)]=\infty$ (see \cite{YYJ});

\item[(S4)] $ \liminf_{|t|\to\infty}\frac{f(x,t)t-4F(x,t)}{|t|^\tau}
>-\alpha$ uniformly in $x\in \Omega$, where $\tau \in [0, 2]$ and
$0< \alpha<a\lambda_1$, $\lambda_1$ is the first eigenvalue of
$\big(-\Delta, H_0^1(\Omega)\big)$ (see \cite{BC}).

\end{itemize}
All kinds of conditions are mainly making sure the boundness of the
Cerami or Palais-Smale sequences. The following condition on $f$
which is called Ambrosetti-Rabinowitz condition  is often used:
\begin{itemize}
\item[(S5)] there exists $\theta>4$ such that
$f(x, t)t\geq \theta F(x,t)$ for $|t|$ large, where $F (x,
t)=\int_0^tf(x, s)ds$.
\end{itemize}

We consider the nonlinear eigenvalue problem
\begin{equation}
\begin{gathered}
-\Big(\int_{\Omega} |\nabla u|^2 dx\Big)\Delta u=\mu u^3,
\quad\text{in }
  \Omega, \\
u=0,  \quad\text{on }  \partial\Omega,
\end{gathered}
\end{equation}
whose the eigenvalues are the critical values of the functional
\begin{equation}
J(u)=\|u\|^4 ,\quad u\in S:=\big\{u \in H_0^1(\Omega): \int_\Omega
|u|^4dx=1\big\},
\end{equation}
where $\|u\|=\big(\int_\Omega |\nabla u|^2dx\big)^{1/2}$. We already
know the first eigenvalue $\mu_1 > 0$ and the first eigenfunction
$\psi_1 > 0$ (see \cite{ZT}).

Now, we can state our main results.

\begin{theorem} \label{thm1.1}
Assume that $f\in C(\Omega\times{\mathbb {R}}, {\mathbb {R}})$
satisfies \eqref{8.1} and
\begin{itemize}
\item[(F1)] ${\lim_{|t|\to\infty}\big(\frac{a\lambda_1}{2}t^2
+\frac{b\mu_1}{4}t^4-F(x,t)\big)=+\infty}$ uniformly in $x\in
\Omega$;

\item[(F2)] there exists $\lambda>\lambda_1$ such that
$F(x,t)\geq\frac{a\lambda}{2}t^2$ for $|t|$ small.
\end{itemize}
Then  \eqref{Eq.(5.1)} has at least one nontrivial solution.
\end{theorem}

\begin{remark} \label{rmk1}  \rm
Theorem \ref{thm1.1} is a new for the case
 $\liminf_{|t|\to\infty}\frac{F(x, t)}{t^4}\leq\frac{b\mu_1}{4}$.
The condition (F1) is  weaker than  (F3) in \cite{YYJ}. So our
theorem is different from their theorems and obtains one nontrivial
solution by adding the condition (F2) near zero.
\end{remark}

\begin{theorem} \label{thm1.2}
Assume that  $f\in C(\Omega\times{\mathbb {R}}, {\mathbb {R}})$
satisfies \eqref{8.1} and
\begin{itemize}
\item[(F3)] ${\liminf_{|t|\to\infty}\frac{F(x, t)}{t^4}>\frac{b\mu_1}{4}}$
 uniformly in $x\in \Omega$;

\item[(F4)] ${\lim_{|t|\to \infty}\big(\frac{1}{4}f(x,t)t-F(x,t)
+\frac{a\lambda_1}{4}t^2\big)=+\infty}$ uniformly in $x\in \Omega$;

\item[(F5)] there exists $\mu<\mu_1$ such that
$F(x,t)\leq\frac{a\lambda_1}{2}t^2+\frac{b\mu}{4}t^4$ for $|t|$
small.
\end{itemize}
Then \eqref{Eq.(5.1)} has at least one nontrivial solution.
\end{theorem}

\begin{remark} \label{rmk2} \rm
Condition (F4) is a new condition for a class of function $f(x,t)$
and is
 weaker than (S1)--(S5).
For example, let
$$
f(x,t)=\frac{a\lambda_1}{8}\Big(8t^3\ln(1+t^2)+\frac{4t^5}{1+t^2}
 +4t^3\cos t^4\Big).
$$
A simple computation shows that
$$
\frac{1}{4}f(x,t)t-F(x,t)+\frac{a\lambda_1}{4}t^2 =
\frac{a\lambda_1}{8}\Big(\frac{t^6 (1+\cos t^4)}{1+t^2} + \frac{t^4
(2+\cos t^4)+2t^2}{1+t^2}- \sin t^4\Big)
$$
and
$$
\lim_{|t|\to
\infty}\Big(\frac{1}{4}f(x,t)t-F(x,t)+\frac{a\lambda_1}{4}t^2\Big)
=+\infty.
$$
Hence, $f(x,t)$ satisfies all the assumptions of Theorem
\ref{thm1.2}, but it does not satisfy any conditions of (S1)--(S5).
\end{remark}

\section{Preliminaries}

We consider  $H:=H_0^1(\Omega)$ endowed with the norm
$\|u\|=\big(\int_\Omega |\nabla u|^2dx\big)^{1/2}$. We denote the
usual $L^p(\Omega)$-norm by $|\cdot|_p$. Since $\Omega $ is a
bounded domain, it is well known that $H \hookrightarrow
L^p(\Omega)$ continuously for
 $p\in[1, 2^*]$, and compactly for $p\in[1, 2^*)$. Moreover there
exists $\gamma_p>0$ such that
\begin{equation}\label{eq6.1}
|u|_p\leq\gamma_p\|u\|, \quad  u\in H.
\end{equation}

Seeking a weak solution of problem \eqref{8.1} is equivalent to
finding a critical point of the $C^1$ functional
\begin{equation}\label{eq6.2}
I(u):=\frac{a}{2}\|u\|^2+\frac{b}{4}\|u\|^4-\int_\Omega F(x,u)dx,
\quad u\in H,
\end{equation}
which implies that
\begin{equation}\label{eq6.3}
\langle I'(u),v \rangle =(a+b\|u\|^2)\int_\Omega \nabla u
\cdot\nabla v dx-\int_\Omega f(x,u)vdx, \ \ u,v\in H.
\end{equation}
Let
$$
{E_j:=\oplus_{i\leq j} \ker(-\Delta-\lambda_i)},
$$
where $0<\lambda_1\leq\lambda_2\leq\lambda_3
\leq\dots\leq\lambda_i\leq\dots$ are the eigenvalues of
$(-\Delta,H)$. We denote a subsequence of a sequence $\{u_n\}$ as
$\{u_n\}$ to simplify the notation unless specified. We need the
following concept, which was introduced by Cerami \cite{GC} and is a
weak version of
 the (PS) condition.

\begin{definition}[\cite{GC}] \label{def2.1} \rm
 Let $J\in C^1(X,\mathbb {R})$, we say that $J$ satisfies the Cerami condition
at the level $c\in \mathbb {R}$ ($(Ce)_c$ for short), if any
sequence $\{u_n\}\subset X$ with
$$
J(u_n)\to c, \quad (1+\|u_n\|)J'(u_n)\to 0\quad\text{as } n\to
\infty,
$$
possesses a convergence subsequence in $X$; $J$ satisfies the $(Ce)$
condition if $J$ satisfies the $(Ce)_c$ for all $c\in \mathbb {R}$.
\end{definition}

The following lemma, which can be found in \cite{DG}, is our main
tool in this article.

\begin{lemma}[Mountain Pass Theorem] \label{lem2.1}
Let $H$ be a real Banach space and $I\in C^1(H, \mathbb R)$
satisfying the $(Ce)$ condition. Suppose $I(0)=0$,
\begin{itemize}
\item[(i)] there are constants $\rho, \beta> 0$ such that
$I|_{\partial B_\rho}\geq\beta$ where
$$
B_\rho=\{u\in H : \|u\|\leq \rho\};
$$

\item[(ii)] there is $u_1\in H$ and $\|u_1\|>\rho$ such that $I(u_1)<0$.
\end{itemize}
Then $I$ possesses a critical value $c\geq \beta$.  Moreover $c$ can
be characterized as
$$
c =\inf_{g\in\Gamma}\max_{u\in g([0,1])} I(u), \quad \Gamma= \{g\in
C([0, 1],H) : g(0) = 0, g(1) = u_1\}.
$$
\end{lemma}

We give a lemma about the $(Ce)$ condition which will play an
important role in the proof of our theorems.

\begin{lemma} \label{lem2.2}
Assume that $f(x,t)$ satisfies \eqref{8.1} and (F4), then $I$
satisfies the $(Ce)$ condition.
\end{lemma}

\begin{proof}
Suppose that $\{u_n\} $ is a $(Ce)_c$ sequence for $c\in \mathbb
{R}$
\begin{equation}\label{eq6.4}
I(u_n)\to c, \quad   (1+\|u_n\|)I'(u_n)\to 0 \quad\text{as } n\to
\infty.
\end{equation}
Now firstly, we prove that $\{u_n\} $ is a bounded sequence. From
\eqref{eq6.2}, \eqref{eq6.3} and \eqref{eq6.4}, we obtain
\begin{equation}\label{eq6.5}
1+c \geq I(u_n)-\frac{1}{4} I'(u_n)u_n
=\frac{a}{4}\|u_n\|^2+\int_\Omega
\Big(\frac{1}{4}f(x,u_n)u_n-F(x,u_n)\Big)dx.
\end{equation}
By (F4), there exists $M>0$ such that
\begin{equation}\label{eq6.6}
\frac{1}{4}f(x,t)t-F(x,t)+\frac{a\lambda_1}{4}|t|^2\geq -M
\end{equation}
for all $x\in\Omega$ and $t\in\mathbb {R}$. And let
$u_n=\phi_n+w_n$, where $\phi_n\in E_1$ and $w_n\in E^{\perp}_1$.
From \eqref{eq6.5} and \eqref{eq6.6}, one obtains
\begin{equation} \label{eq6.7}
\begin{aligned}
1+c
&\geq I(u_n)-\frac{1}{4} I'(u_n)u_n  \\
&=\frac{a}{4}\|u_n\|^2-\frac{a\lambda_1}{4}|u_n|^2_2 +\int_\Omega
\Big(\frac{1}{4}f(x,u_n)u_n-F(x,u_n)
 +\frac{a\lambda_1}{4}|u_n|^2\Big)dx \\
&\geq
\frac{a}{4}\big(1-\frac{\lambda_1}{\lambda_2}\big)\|w_n\|^2-M|\Omega|
\end{aligned}
\end{equation}
which implies that $\|w_n\|$ is bounded. We claim that $\{u_n\}$ is
a bounded sequence. Otherwise, there is a subsequence of $\{u_n\}$
satisfying $\|u_n\| \to+\infty$ as $n\to+\infty$. Then we obtain
 $$
\frac{w_n}{\|u_n\|} \to 0 \in H.
$$
Since $\phi_n/\|u_n\|$ is bounded in $E_1$ ($E_1$ has finite
dimension), we have $\phi_n/\|u_n\|\to v$ in $E_1$. By
$$
v_n:=\frac{u_n}{\|u_n\|}=\frac{\phi_n+w_n}{\|u_n\|}
=\frac{\phi_n}{\|u_n\|}+\frac{w_n}{\|u_n\|}\to v\in  E_1,
$$
one has
\begin{equation}\label{eq6.8}
\frac{u_n(x)}{\|u_n\|}\to v(x)\quad\text{a.e. in }\Omega.
\end{equation}
 From $\|v_n\|=1$, we obtain that $\|v\|=1$. And by $v\in E_1$, one has
 that $v(x)>0$ or $v(x)<0$, which implies that
\begin{equation}\label{eq6.9}
|u_n(x)|\to +\infty\quad\text{as }n\to+\infty
\end{equation}
for all $x\in \Omega$ by \eqref{eq6.8}. It follows from
\eqref{eq6.7}, \eqref{eq6.9} and Fatou's lemma that
\begin{align*}
1+c
&\geq I(u_n)-\frac{1}{4} I'(u_n)u_n  \\
&=\frac{a}{4}\|u_n\|^2+\int_\Omega \Big(\frac{1}{4}f(x,u_n)u_n-F(x,u_n)\Big)dx \\
&\geq \int_\Omega
\Big(\frac{1}{4}f(x,u_n)u_n-F(x,u_n)+\frac{a\lambda_1}{4}|u_n|^2
\Big)dx \\
&\to+\infty\quad\text{as }n \to +\infty,
\end{align*}
which is a contradiction. Then we get that $\{u_n\}$ is bounded in
$H$. Since $f(x,t)$ is subcritical growth, we can easily obtain that
$\{u_n\}$ has a convergence subsequence. Hence, $I$ satisfies the
$(Ce)$ condition.
\end{proof}

\section{Proof of main results}

\begin{proof}[Proof of Theorem \ref{thm1.1}]
 Let
$$
\overline{u}=\Big(\int_\Omega \nabla u \cdot \nabla \phi_1 dx
\Big)\phi_1, \quad \widetilde{u}=u-\overline{u},
$$
where the $\phi_1$ is the first eigenfunction corresponding to
$\lambda_1$.

The following statements come from \cite{SQ}. First, there exist a
real function $g\in L^1(\Omega)$, and $G\in C(\mathbb{R},\
\mathbb{R})$ which is subadditive; that is,
$$
G(s+t)\leq G(s)+G(t)
$$
for all $s,\ t\in \mathbb{R}$, and coercive; that is,
$G(t)\to+\infty$ as $|t|\to\infty$, and satisfies
$$
G(t)\leq|t|+4
$$
for all $t\in \mathbb{R}$, such that
$$
F(x, t)-\frac{a\lambda_1}{2}t^2-\frac{b\mu_1}{4}t^4\leq-G(t)+g(x)
$$
for all $t\in \mathbb{R}$ and $x\in \Omega$.

Second, the functional $\int_\Omega G(v)dx$ is coercive on $E_1$
(this result also can be seen in \cite{SQ}). We claim that $I(u)$ is
coercive.
\begin{align*}
&\int_\Omega \Big( F(x, u)-\frac{a\lambda_1}{2}u^2-\frac{b\mu_1}{4}u^4 \Big)dx\\
&\leq -\int_\Omega G(u)dx+\int_\Omega g(x)dx \\
& \leq -\int_\Omega\left( G(\overline{u})-G(-\widetilde{u})\right)
dx
 +\int_\Omega g(x)dx \\
& \leq -\int_\Omega G(\overline{u})dx +|\widetilde{u}|_{1}+4
|\Omega|
 +\int_\Omega g(x)dx \\
& \leq -\int_\Omega G(\overline{u})dx +C_1(\|\widetilde{u}\|+1)
\end{align*}
for all $u\in H$ and some
$$
C_1=C+4  |\Omega|+\int_\Omega g(x)dx,
$$
where $C$ is a positive constant in Sobolev's inequality,
$$
|u|_{1}\leq C\|u\|,\quad |u|_{2}\leq C\|u\|
$$
for all $u\in H$. Hence we have
\begin{align*}
I(u_n)
&=\frac{a}{2}\|u\|^2+\frac{b}{4}\|u\|^4-\int_\Omega F(x,u)dx \\
&=\frac{a}{2}\|u_n\|^2-\frac{a\lambda_1}{2}|u_n|^2_2
 +\frac{b}{4}\|u_n\|^4-\frac{b\mu_1}{4}|u_n|^4_4 \\
&\quad+\int_\Omega\Big(\frac{a\lambda_1}{2}|u_n|^2+\frac{b\mu_1}{4}|u_n|^4
 -F(x,u_n)\Big)dx \\
&\geq\frac{a}{2}\|u_n\|^2-\frac{a\lambda_1}{2}|u_n|^2_2
 +\int_\Omega\Big(\frac{a\lambda_1}{2}|u_n|^2+\frac{b\mu_1}{4}|u_n|^4-F(x,u_n)\Big)dx \\
&\geq \frac{a}{2}\big(1-\frac{\lambda_1}{\lambda_2}\big)
 \|\widetilde{u}\|^2+\int_\Omega G(\overline{u})dx-C_1(\|\widetilde{u}\|+1)
\end{align*}
for all $u\in H$. By the coercivity of the functional $\int_\Omega
G(v)dx$ on $E_1$ and that fact
$$
\|u\|^2=\|\overline{u}\|^2+\|\widetilde{u}\|^2,
$$
which implies that the functional $I(u)$ is coercive. $I$ satisfies
the $(Ce)$ condition and is bounded from below. By (F2), we have
$$
F(x,t)\geq \frac{a\lambda}{2}t^2-C|t|^p
$$
for all $x\in \Omega$ and $t\in\mathbb{R}$, which implies that
\begin{align*}
I(u)
&\leq\frac{a}{2}\|u\|^2+\frac{b}{4}\|u\|^4-\frac{a\lambda}{2}|u|_2^2+C|u|_p^p \\
&=\frac{a}{2}\big(1-\frac{\lambda}{\lambda_1}\big)\|u\|^2
+\frac{b}{4}\|u\|^4+C\|u\|^p <0
\end{align*}
for $u\in E_1\cap B_\delta$,  $\lambda>\lambda_1$, where $\delta>0$
small enough and $E_1$ is the subspace of $H$ spanned by $\phi_1$
the eigenfunctions of $\lambda_1$. Then $I(u)$ achieves the negative
infimum.
 This completes the proof
\end{proof}

\begin{proof}[Proof of Theorem \ref{thm1.2}]
By Lemmas \ref{lem2.1} and \ref{lem2.2}, it is sufficient to show
that $I$ satisfies (i) and (ii).
\smallskip

\noindent\textbf{Step 1.} There are constants $\rho,\ \beta> 0$ such
that $I(u)\geq\beta$ for all $\|u\|=\rho$. In fact, by (F5), it is
easy to see that
$$
F(x,t)\leq\frac{a\lambda_1}{2}t^2+\frac{b(\mu_1-\varepsilon)}{4}t^4+C|t|^p
$$
for all $t\in \mathbb{R}$ and $x\in \Omega$,
\begin{align*}
I(u)&\geq\frac{a}{2}\|u\|^2+\frac{b}{4}\|u\|^4
 -\frac{a\lambda_1}{2}|u|_2^2-\frac{b(\mu_1-\varepsilon)}{4}|u|_4^4
 -C\int_\Omega |u|^pdx \\
&\geq\frac{b}{4}\big(1-\frac{\mu_1-\varepsilon}{\mu_1}\big)\|u\|^4-C\gamma_p\|u\|^p.
\end{align*}
Note that $4< p < 2^*$, then for $\varepsilon$ small enough. So
there exists $\beta>0$ such that $I(u)\geq\beta$ for all
$\|u\|=\rho$, where $\rho>0$ small enough.
\smallskip

\noindent\textbf{Step 2.} There exists $u_1\in H$ and $\|u_1\|>\rho$
such that
 $I(u_1)<0$.
Indeed, for small $\varepsilon > 0$, by the definition of $\mu_1$,
we can choose $u\in S$ satisfying
 \begin{equation}\label{eq7.6}
\mu_1+\frac{\varepsilon}{2}\geq\|u\|^4.
\end{equation}
It follows from (F3) that
\begin{equation}\label{eq7.7}
F(x,t)\geq\frac{b(\mu_1+\varepsilon)}{4}t^4-C.
\end{equation}
Hence, by \eqref{eq7.6} and \eqref{eq7.7}, we have
\begin{equation} \label{eq7.8}
\begin{aligned}
I(tu) &\leq\frac{a}{2}t^2\|u\|^2+\frac{b}{4}t^4\|u\|^4
 -\frac{b}{4}t^4(\mu_1+\varepsilon)+C|\Omega| \\
&\leq\frac{a}{2}t^2\|u\|^2+\frac{b}{4}t^4\mu_1
 +\frac{b\varepsilon}{8}t^4-\frac{b}{4}t^4(\mu_1+\varepsilon)+C|\Omega| \\
&=-\frac{b\varepsilon}{8}t^4+\frac{a}{2}t^2\|u\|^2+C|\Omega|.
\end{aligned}
\end{equation}
Thus, $I(tu)\to-\infty$ as $t\to\infty$. Therefore, there is $u_1
\in H$ with $\|u_1\|>\rho$ such that $I(u_1)< 0$.
 This completes the proof.
\end{proof}

\end{document}